\documentclass[12pt,reqno]{amsart}

\usepackage{amssymb}
\usepackage{amscd}
\usepackage{amsfonts}
\usepackage{setspace}
\usepackage{version}

\numberwithin{equation}{section}
   \theoremstyle{plain}
\newtheorem{theorem}{Theorem}
\newtheorem{proposition}{Proposition}

\newtheorem{corollary}{Corollary}

\newtheorem*{mainthm3-repeat}{Theorem \ref{mainthm3}}

\renewcommand{\leq}{\leqslant}
\renewcommand{\geq}{\geqslant}
\newsavebox{\proofbox}
\savebox{\proofbox}{\begin{picture}(7,7)  \put(0,0){\framebox(7,7){}}\end{picture}}

\newcommand\Z{\mathbb{Z}}
\newcommand\R{\mathbb{R}}

\newcommand\eps{\varepsilon}

\def\endproof{\hfill{\usebox{\proofbox}}\vspace{11pt}}

\onehalfspace

\begin{document}

\title{Contractions and Expansion}

\author{Emmanuel Breuillard}
\address{Laboratoire de Math\'ematiques\\
B\^atiment 425, Universit\'e Paris Sud 11\\
91405 Orsay\\
FRANCE}
\email{emmanuel.breuillard@math.u-psud.fr}

\author{Ben Green}
\address{Centre for Mathematical Sciences\\
Wilberforce Road\\
Cambridge CB3 0WA\\
England }
\email{b.j.green@dpmms.cam.ac.uk}

\maketitle

\begin{abstract}
Let $A \subseteq \R$ be a finite set and let $K \geq 1$ be a real number. Suppose that for each $a \in A$ we are given an injective map $\phi_a : A \rightarrow \R$ which fixes $a$ and contracts other points towards it in the sense that $|a - \phi_a(x)| \leq \frac{1}{K} |a - x|$ for all $x \in A$, and such that $\phi_a(x)$ always lies between $a$ and $x$. Then
\[ |\bigcup_{a \in A} \phi_a (A)| \geq \frac{K}{10} |A| - O_K(1).\]
An immediate consequence of this is the estimate $|A + K \cdot A| \geq \frac{K}{10} |A| - O_K(1)$, which is a slightly weakened version of a result of Bukh.
\end{abstract}

\begin{center} \emph{To the memory of Yayha Hamidoune} \end{center}

\section{Introduction}

In this short note we consider the behaviour of a set $A \subseteq \R$ under a collection of maps $\phi_a : A \rightarrow \R$. Let $K \geq 1$ be a parameter. We will assume that these maps have the following properties:

\begin{enumerate}
\item $\phi_a$ is injective;
\item $\phi_a(a) = a$;
\item $\phi_a$ is a $K$-contraction in the sense that $|a - \phi_a(x)| \leq \frac{1}{K} |a - x|$ for all $x \in A$;
\item $\phi_a(x)$ lies between $a$ and $x$.
\end{enumerate}

\begin{theorem}\label{mainthm}
Suppose that $A \subseteq \R$ is a finite set of size $n$ and that we have maps $\phi_a$ as above.  Then \[ |\bigcup_{a \in A} \phi_a(A) | \geq \textstyle\frac{1}{8} \displaystyle Kn (1 - n^{-1/K^2}).\]
\end{theorem}

\emph{Remark.} The bound we have given here looks a little odd, but it is convenient for our proof. Note that it is at least $\frac{1}{10}Kn - O(e^{CK^2})$, a slightly more precise version of the bound stated in the abstract.

An immediate corollary of this theorem is the following. Here, $A + K \cdot A := \{ a + Ka' : a, a' \in A\}$.

\begin{corollary}\label{bukh-cor}
Suppose that $A \subseteq \R$ is a finite set and that $K \geq 1$ is a real number. Then $|A + K \cdot A| \geq \frac{1}{10}K|A| - O(e^{CK^2})$.
\end{corollary}
\begin{proof} Simply apply the theorem with $\phi_a(x) := (x + Ka)/(K+1)$. These maps obviously verify (i) -- (iv) above.\end{proof}

We note that Bukh \cite{bukh} established a much more precise result when $K \in \Z$, namely that $|A + K \cdot A| \geq (K+1)|A| - o(|A|)$. Assuming that $K$ is an integer should not make things any easier, and furthermore our approach would appear not to generalise to the more general sums of dilates $\lambda_1 \cdot A + \dots + \lambda_t \cdot A$ considered by Bukh. Let us also note that Cilleruelo, Hamidoune and Serra \cite{chs} obtained an extremely precise result when $K$ is prime, establishing that $|A + K \cdot A| \geq (K+1)|A| - \lceil K(K+2)/4\rceil$ for $|A| \geq 3(K-1)^2 (K-1)!$.

\section{Proof of the main theorem}

In this section we prove Theorem \ref{mainthm}. Let $F(n)$ be the minimum size of $\bigcup_{a \in A} \phi_a(A)$ over all sets $A$ of size $n$. We will obtain a lower bound for $F(n)$ in terms of the values of $F(n')$, $n' < n$; we may then proceed by induction.

We will use the (obvious)  \emph{convexity} property of maps $\phi_a$ satisfying (i) -- (iv) above, namely that $\phi_a(I) \subseteq I$ for any interval $I$ containing $a$.

We clearly have $F(1) = 1$, so suppose that $A \subseteq \R$ is a set of size $n \geq 2$. We may rescale so that the extreme points of $A$ are $0$ and $1$. Suppose that there is some $a_* \in A$ such that $|A \cap [a_* - 1/K, a_* + 1/K|| \leq 6n/K$. Write $A_1 := A \cap [0, a_* - 1/K)$ and $A_2 := A \cap (a_* + 1/K, 1]$, and set $n_1 := |A_1|$, $n_2 := |A_2|$. Then $\phi_{a_*}$ contracts all of $A$ into the interval $[a_* - 1/K, a_* + 1/K]$. By induction and the convexity property we have
\begin{align}\nonumber F(n) & \geq |\bigcup_{a \in A_1} \phi_a(A_1)| + |A| + |\bigcup_{a \in A_2} \phi_a (A_2)| \\ & \geq F(n_1) + n + F(n_2).\label{case1}\end{align}
Note that $n_1 + n_2 \geq (1 - 6/K) n$. 

Alternatively, suppose there is no such $a_*$. Obviously $A = A \cap \bigcup I_a$, where $I_a = [a - 1/K, a+ 1/K]$. We may pass to disjoint subcollections $\bigcup_{a \in S_1} I_a$ and $\bigcup_{a \in S_2} I_a$ whose union is $\bigcup_{a \in A} I_a$ (cf. \cite{croft}). By assumption, $|A \cap I_a| > 6n/K$, and therefore $|S_1|, |S_2| < K/6$. It follows that $A$ is covered by $< K/3$ intervals of length $2/K$, and hence there is some $a^* \in A$, $a^* < 1$, such that $A$ is disjoint from $(a^*, a^* + 1/K]$. Set $A_1 := A \cap [0, a^*]$ and $A_2 := (a^* + 1/K, 1]$, and set $n_1 := |A_1|$, $n_2 := |A_2|$; note that $n_1 + n_2 = n$.  Note also that $\phi_{a^*}(A_2) \subseteq (a^*, a^* + 1/K]$; here, we make crucial use of property (iv), which asserts that $\phi_{a_*}(x)$ lies between $a_*$ and $x$.

By the convexity property and the above observations,
\[ F(n) \geq |\bigcup_{a \in A_1} \phi_a(A_1)| + |A_2| + |\bigcup_{a \in A_2} \phi_a (A_2)| \geq |A_2| + F(n_1) + F(n_2).\]
Note, however, that $A_2$ contains $1$ and hence $A \cap I_1$, a set of size $> 6n/K$. Therefore
\begin{equation}\label{case2} F(n) \geq \frac{6n}{K} + F(n_1) + F(n_2)\end{equation} in this case.

We have established that for each $n$ there are $n_1, n_2 < n$ such that either \eqref{case1} or \eqref{case2} holds. That is, 

\begin{align}\nonumber F(n) \geq \min \big( \min_{\substack{n_1, n_2 < n \\ n_1 + n_2 = n}} & (F(n_1) + F(n_2) + \frac{6n}{K}), \\ & \min_{\substack{n_1, n_2 < n \\  n_1 + n_2 \geq n(1 - 6/K)}} (F(n_1) + F(n_2) + n) \big).\label{inequalities}\end{align}

It remains to verify, by induction on $n$, that $F(n) \geq \frac{1}{8} Kn (1 - n^{-1/K^2})$. 

Dealing with the first inequality immediately reduces to showing that
\[ \frac{48 n}{K^2} \geq n_1^{1 - 1/K^2} + n_2^{1-1/K^2} - n^{1 - 1/K^2}\] whenever $n_1 + n_2 = n$. But the largest value of the right-hand side is never more than when $n_1 = n_2 = n/2$, and it is then enough to note that $48/K^2 \geq 2^{1/K^2} - 1$. 

Checking the second inequality is even easier and amounts, using the inequality $n_1 + n_2 \geq (1 - 6/K) n$, to establishing that
\[ \frac{2n}{K} \geq n_1^{1-1/K^2} + n_2^{1- 1/K^2} - n^{1 - 1/K^2}\] under the assumption that $n_1 + n_2 \leq n$.  This completes the proof of Theorem \ref{mainthm}.\endproof

\section{Some further remarks.}

Condition (iv) above, that $\phi_a(x)$ always lies between $a$ and $x$, may seem a little unnatural.  In this section we note that Theorem \ref{mainthm} fails without this assumption.

\begin{proposition}
Let $K \geq 1$ be arbitrary. Then there are arbitrarily large finite sets $A$ together with collections of maps $\phi_a : A \rightarrow \R$ satisfying \textup{(i)}, \textup{(ii)} and \textup{(iii)} above but such that 
\[ |\bigcup_{a \in A} \phi_a (A)| = 2|A|.\]
\end{proposition}
\begin{proof}
Assume that $K > 10^4$ (say). Consider the set
\[ A := \{ \eps_n (4K)^{-n} + \dots + \eps_1 (4 K)^{-1} : \eps_1,\dots,\eps_n \in \{0,1\}\},\]
and define 
\[ \tilde A := \{ \eps_{n+1} (4 K)^{-n-1} + \dots + \eps_1 (4 K)^{-1} : \eps_1,\dots, \eps_{n+1} \in \{0,1\}\}.\]
Obviously $|\tilde A| = 2 |A|$. There is an obvious map $\phi_0 : A \rightarrow \tilde A$ defined by $\phi_0(x) = x/4K$; this clearly satisfies $\phi_0(0) = 0$, and is a $1/4K$-contraction. Now for each $x \in A$ there is a bijection $\psi_x : \tilde A \rightarrow \tilde A$ with $\psi_x(A) = A$, $\psi_x(0) = x$, and which is $2$-biLipschitz in the sense that 
\[ \textstyle\frac{1}{2}|t - t'| \leq |\psi_x(t) - \psi_x(t')| \leq 2|t - t'|\] for all $t, t' \in \tilde A$. Such a map may be constructed by viewing $\tilde A$ as the set of vertices of a binary tree of depth $n+1$ and then applying a suitable binary tree automorphism. The distance between two nodes is determined, up to a factor of at most $1.01$, by the point in the tree at which they branch (or equivalently the smallest $m$ for which $\eps_m \neq \eps'_m$), and this is preserved by any tree automorphism.

Now define $\phi_x : A \rightarrow \tilde A$ by $\phi_x := \psi_x \circ \phi_0 \circ \psi_x^{-1}$. This is well-defined because $\psi_x^{-1}(A) = A$, and $\phi_0$ is defined on $A$. Obviously $\phi_x(x) = x$. Furthermore $\phi_x$, being a composition of a $1/4K$-contraction and two $2$-biLipschitz maps, is a $1/K$-contraction. Each $\phi_x$ is injective, and finally $|\bigcup_{a \in A} \phi_a(A)| = |\tilde A| = 2|A|$. This completes the proof.\end{proof}

It is clear that one cannot remove the factor of $2$ in this proposition if $K \geq 2$. Indeed if $a_1,a_2$ are the extreme points of $A$ then $\phi_{a_1}(A)$ and $\phi_{a_2}(A)$ are disjoint and both have cardinality $|A|$.

Finally let us note that one should not expect a version of Corollary \ref{bukh-cor} with polynomial dependencies throughout. Indeed the set 
\[ A := \{ n_0 + n_1 K + \dots + n_{d-1} K^{d-1} : n_0,\dots, n_{d-1} \in [N]\}\]
has $|A + K \cdot A| \leq 2^d N^{d+1} = 2^{d} N |A|$. Taking $d \sim \frac{1}{2}\log_2 L$ and $N := L/2^d$ we may ensure that $|A + K \cdot A| \leq L|A|$ whilst $|A| = N^d \geq e^{c (\log  L)^2}$. In particular, if $L \sim e^{(\log K)^{3/4}}$ then we have $|A + K \cdot A| \leq K^{o(1)} |A|$ whilst $|A| \geq K^{c(\log K)^{1/2}}$.

\section{Acknowledgement}

The authors would like to thank Boris Bukh for reading an earlier draft of this paper.

\end{document}